\documentclass[a4paper,11pt,fleqn]{article}
\usepackage[final]{showkeys}
\usepackage[obeyspaces,hyphens,spaces]{url}
\renewcommand*\showkeyslabelformat[1]{%
\fbox{\parbox[t]{1.4 cm}{\raggedright\normalfont\tiny%
\url{#1}}}}
\usepackage[dvipsnames]{xcolor}
\usepackage{enumitem}
\usepackage{amsmath}
\usepackage{amssymb}
\usepackage{amsfonts}
\usepackage{theorem}
\usepackage{mathtools}
\usepackage{euscript}
\usepackage{exscale,relsize}
\usepackage{epic,eepic}
\usepackage{multicol}
\usepackage{makeidx}
\usepackage{srcltx} 
\usepackage{charter}
\usepackage{etoolbox}
\usepackage{mathrsfs} 
\usepackage{graphicx}
\usepackage{authblk}
\newcommand{\email}[1]{\href{mailto:#1}{\nolinkurl{#1}}}
\usepackage[normalem]{ulem} 
\usepackage{epsfig}
\usepackage[usenames,dvipsnames]{pstricks}
\usepackage{pstricks-add}
\usepackage{pst-grad}
\usepackage{pst-plot} 
\usepackage{pst-eucl} 
\definecolor{color2b}{rgb}{0.6,0.8,1.0}
\allowdisplaybreaks 
\newlength{\mySubFigSize}
\definecolor{labelkey}{rgb}{0,0.08,0.45}
\definecolor{refkey}{rgb}{0,0.6,0.0}
\definecolor{dblue}{HTML}{0455BF}
\definecolor{dgreen}{HTML}{02724A}
\definecolor{Dblue}{HTML}{8602DC}
\definecolor{dred}{HTML}{D90404}
\usepackage{upref}
\usepackage{hyperref}
\hypersetup{
colorlinks=true,
linktocpage=true,
linkcolor=dblue,
citecolor=dgreen,
urlcolor=dred}
\renewcommand{\leq}{\ensuremath{\leqslant}}
\renewcommand{\geq}{\ensuremath{\geqslant}}

\newcommand{\abscal}[2]{\left|\left\langle{{#1}\mid{#2}}%
\right\rangle\right|} 
\newcommand{\scal}[2]{{\langle{{#1}\mid{#2}}\rangle}}
\newcommand{\sscal}[2]{{\big\langle{{#1}\mid{#2}}\big\rangle}}
\newcommand{\pair}[2]{\langle{{#1},{#2}}\rangle} 
 
\newcommand{\Pair}[2]{\big\langle{{#1},{#2}}\big\rangle} 
\newcommand{\menge}[2]{\big\{{#1}~|~{#2}\big\}} 
 
\newcommand{\norma}[1]{{#1}^{\mbox{\tiny$\sharp$}}}
\newcommand{\HH}{\ensuremath{{\mathcal X}}}

\newcommand{\XX}{\ensuremath{{\mathcal X}}}
\newcommand{\YY}{\ensuremath{{\mathcal Y}}}
\newcommand{\GG}{\ensuremath{{\mathcal Y}}}
\newcommand{\KK}{\ensuremath{{\mathcal Z}}}
\newcommand{\Sum}{\ensuremath{\displaystyle\sum}}

\newcommand{\emp}{\ensuremath{\varnothing}}

\newcommand{\Id}{\ensuremath{\operatorname{Id}}}

\newcommand{\RR}{\ensuremath{\mathbb{R}}}
\newcommand{\RP}{\ensuremath{\left[0,{+}\infty\right[}}

\newcommand{\BL}{\ensuremath{\EuScript B}}

\newcommand{\RPP}{\ensuremath{\left]0,{+}\infty\right[}}

\newcommand{\RPX}{\ensuremath{\left[0,{+}\infty\right]}}
\newcommand{\RX}{\ensuremath{\left]{-}\infty,{+}\infty\right]}}

\newcommand{\NN}{\ensuremath{\mathbb N}}
\newcommand{\ZZ}{\ensuremath{\mathcal Z}}
\newcommand{\intdom}{\ensuremath{\operatorname{int}\operatorname{dom}}}

\newcommand{\weakly}{\ensuremath{\:\rightharpoonup\:}}
\newcommand{\exi}{\ensuremath{\exists\,}}

\newcommand{\ran}{\ensuremath{\operatorname{ran}}}

\newcommand{\zer}{\ensuremath{\operatorname{zer}}}
\newcommand{\pinf}{\ensuremath{{{+}\infty}}}

\newcommand{\dom}{\ensuremath{\operatorname{dom}}}

\newcommand{\prox}{\ensuremath{\operatorname{prox}}}
\newcommand{\proj}{\ensuremath{\operatorname{proj}}}
\newcommand{\Fix}{\ensuremath{\operatorname{Fix}}}

\newcommand{\gra}{\ensuremath{\operatorname{gra}}}

\newcommand{\zeroun}{\ensuremath{\left]0,1\right[}}

\newtheorem{theorem}{Theorem}[section]
\newtheorem{lemma}[theorem]{Lemma}
\newtheorem{corollary}[theorem]{Corollary}
\newtheorem{proposition}[theorem]{Proposition}

\theoremstyle{plain}{\theorembodyfont{\rmfamily}%
}

\theoremstyle{plain}{\theorembodyfont{\rmfamily}%
\newtheorem{example}[theorem]{Example}}
\theoremstyle{plain}{\theorembodyfont{\rmfamily}%
\newtheorem{remark}[theorem]{Remark}}
\theoremstyle{plain}{\theorembodyfont{\rmfamily}%
}
\theoremstyle{plain}{\theorembodyfont{\rmfamily}%
}
\theoremstyle{plain}{\theorembodyfont{\rmfamily}%
\newtheorem{definition}[theorem]{Definition}}
\theoremstyle{plain}{\theorembodyfont{\rmfamily}%
\newtheorem{problem}[theorem]{Problem}}
\theoremstyle{plain}{\theorembodyfont{\rmfamily}%
}

\numberwithin{equation}{section}
\let\to\rightarrow
\newcommand*\mute{{\mkern 2mu\cdot\mkern 2mu}}
\usepackage{caption}
\begin{document}

\title{\sffamily\huge%
\vskip -12mm 
Warped Proximal Iterations for Monotone Inclusions\thanks{Contact 
author: P. L. Combettes, \email{plc@math.ncsu.edu},
phone: +1 (919) 515 2671.
This work was supported by the National Science Foundation under
grant DMS-1818946.}}

\author{Minh N. B\`ui and Patrick L. Combettes\\
\small North Carolina State University,
Department of Mathematics,
Raleigh, NC 27695-8205, USA\\
\small \email{mnbui@ncsu.edu}\: and \:\email{plc@math.ncsu.edu}
}
\date{~}
\maketitle
\noindent
{\bfseries Abstract.}
Resolvents of set-valued operators play a central role in various
branches of mathematics 
and in particular in the design and the analysis of
splitting algorithms for solving monotone inclusions. We
propose a generalization of this notion, called warped
resolvent, which is constructed with the help of an auxiliary
operator. The properties of warped resolvents are investigated and
connections are made with existing notions. 
Abstract weak and strong convergence principles based on warped
resolvents are proposed and shown to not only provide a synthetic
view of splitting algorithms but to also constitute an effective 
device to produce new solution methods for challenging inclusion
problems. 



\section{Introduction}

A generic problem in nonlinear analysis and optimization is to find
a zero of a maximally monotone operator $M\colon\HH\to 2^{\HH}$,
where $\HH$ is a real Hilbert space.
The most elementary method designed for this task 
is the proximal point algorithm \cite{Roc76a}
\begin{equation}
\label{e:prox1}
(\forall n\in\NN)\quad x_{n+1}=J_{\gamma_n M}x_n, 
\quad\text{where}\quad\gamma_n\in\RPP
\quad\text{and}\quad J_{\gamma_n M}=(\Id+\gamma_n M)^{-1}.
\end{equation}
In practice, the execution of \eqref{e:prox1} may be hindered by 
the difficulty of evaluating the resolvents
$(J_{\gamma_nM})_{n\in\NN}$.
Thus, even in the simple case when $M$ is the sum of two
monotone operators $A$ and $B$, there is no mechanism to 
express conveniently the resolvent of $M$ in terms of operators
involving $A$ and $B$ separately. To address this issue, various
splitting strategies have been proposed to handle increasingly
complex formulations in which $M$ is a composite operator assembled
from several elementary blocks that can be linear operators and
monotone operators 
\cite{Bane20,Livre1,Botr19,Bot13d,Bot13c,Bric15,Siop13,MaPr18,Svva12,
Cond13,Ragu19,Bang13}. In the present paper, we explore a different
path by placing at the core of our analysis the following extension
of the classical notion of a resolvent.

\begin{definition}[Warped resolvent]
\label{d:wr}
Let $\XX$ be a reflexive real Banach space with topological dual
$\XX^*$, let $D$ be a nonempty subset of $\XX$,
let $K\colon D\to\XX^*$, and let $M\colon\XX\to 2^{\XX^*}$
be such that $\ran K\subset\ran(K+M)$ and $K+M$ is injective
(see Definition~\ref{d:inj}).
The warped resolvent of $M$ with kernel $K$ is
$J_M^K=(K+M)^{-1}\circ K$. 
\end{definition}

A main motivation for introducing warped resolvents is that,
through judicious choices of a kernel $K$ tailored to the structure
of an inclusion problem, one can create simple patterns to design
and analyze new, flexible, and modular splitting algorithms. 
At the same time, the theory required to analyze
the static properties of warped resolvents as
nonlinear operators, as well as the dynamics of algorithms using
them, needs to be
developed as it cannot be extrapolated from the classical case,
where $K$ is simply the identity operator. In the present paper, 
this task is undertaken and we illustrate the pertinence of warped
iteration methods through applications to challenging
monotone inclusion problems.

The paper is organized as follows. Section~\ref{sec:2} is dedicated
to notation and background.
In Section~\ref{sec:3}, we provide important illustrations of
Definition~\ref{d:wr} and make connections with constructions 
found in the literature. The properties of warped resolvents are
also discussed in that section. Weakly and strongly convergent
warped proximal iteration methods are introduced and analyzed in
Section~\ref{sec:4}. Besides the use of kernels varying at each
iteration, our framework also features evaluations of warped
resolvents at points that may not be the current iterate, which
adds considerable flexibility and models in particular inertial
phenomena and other perturbations. New splitting algorithms
resulting from the proposed warped iteration constructs are devised
in Section~\ref{sec:5} to solve monotone inclusions.

\section{Notation and background}
\label{sec:2}

Throughout the paper, $\HH$, $\GG$, and $\KK$ are reflexive real
Banach spaces. We denote the canonical pairing between 
$\XX$ and its topological dual $\XX^*$ by $\pair{\cdot}{\cdot}$,
and by
$\Id$ the identity operator. The weak convergence of a sequence
$(x_n)_{n\in\NN}$ to $x$ is denoted by $x_n\weakly x$, while
$x_n\to x$ denotes its strong convergence.
The space of bounded linear operators from $\HH$ to $\GG$ is
denoted by $\BL(\HH,\GG)$, and we set $\BL(\HH)=\BL(\HH,\HH)$.

Let $M\colon\HH\to 2^{\HH^*}$. We denote 
by $\gra M=\menge{(x,x^*)\in\HH\times\HH^*}{x^*\in Mx}$ the graph
of 
$M$, by $\dom M=\menge{x\in\HH}{Mx\neq\emp}$ the domain of $M$, 
by $\ran M=\menge{x^*\in\HH^*}{(\exi x\in\HH)\;x^*\in Mx}$ 
the range of $M$,
by $\zer M=\menge{x\in\HH}{0\in Mx}$ the set of zeros of $M$, 
and by $M^{-1}$ the inverse of $M$, i.e., 
$\gra M^{-1}=\menge{(x^*,x)\in\HH^*\times\HH}{x^*\in Mx}$. 
Further, $M$ is monotone if
\begin{equation}
\big(\forall (x,x^*)\in\gra M\big)
\big(\forall (y,y^*)\in\gra M\big)\quad
\pair{x-y}{x^*-y^*}\geq 0,
\end{equation}
and maximally monotone if, in addition, there exists no
monotone operator $A\colon\HH\to 2^{\HH^*}$ such that
$\gra M\subset\gra A\neq\gra M$.
We say that $M$ is uniformly monotone
with modulus $\phi\colon\RP\to\RPX$ if 
$\phi$ is increasing, vanishes only at $0$, and 
\begin{equation}
\label{e:umonotone}
\big(\forall (x,x^*)\in\gra M\big)
\big(\forall (y,y^*)\in\gra M\big)\quad
\pair{x-y}{x^*-y^*}\geq\phi\big(\|x-y\|\big).
\end{equation}
In particular, $M$ is strongly monotone with constant
$\alpha\in\RPP$ if it is uniformly
monotone with modulus $\phi=\alpha|\mute|^2$.

\begin{definition}
\label{d:inj}
An operator $M\colon\HH\to 2^{\HH^*}$ is injective if 
$(\forall x\in\HH)(\forall y\in\HH)$
$Mx\cap My\neq\emp$ $\Rightarrow$ $x=y$.
\end{definition}

The following lemma, which concerns a type of duality for monotone
inclusions studied in \cite{Joca16,Penn00,Robi99}, will be
instrumental.

\begin{lemma}
\label{l:5498}
Let $A\colon\GG\to 2^{\GG^*}$ and $B\colon\KK\to 2^{\KK^*}$
be maximally monotone, let $L\in\BL(\GG,\KK)$,
let $s^*\in\GG^*$, and let $r\in\KK$.
Suppose that $\HH=\GG\times\KK\times\KK^*$ (hence
$\HH^*=\GG^*\times\KK^*\times\KK$), define
\begin{equation}
\label{e:1571}
M\colon\HH\to 2^{\HH^*}\colon(x,y,v^*)\mapsto
(-s^*+Ax+L^*v^*)\times(By-v^*)\times\{r-Lx+y\},
\end{equation}
and set
$Z=\menge{(x,v^*)\in\GG\times\KK^*}{s^*-L^*v^*\in Ax\;\text{and}\;
Lx-r\in B^{-1}v^*}$.
In addition, denote by $\mathscr{P}$ the set of solutions
to the primal problem
\begin{equation}
\label{e:p4}
\text{find}\;\: x\in\GG\;\:\text{such that}\;\:
s^*\in Ax+L^*\big(B(Lx-r)\big),
\end{equation}
and by $\mathscr{D}$ the set of solutions to the dual problem
\begin{equation}
\label{e:d4}
\text{find}\;\:v^*\in\KK^*\;\:\text{such that}\;\:
-r\in {-}L\big(A^{-1}(s^*-L^*v^*)\big)+B^{-1}v^*.
\end{equation}
Then the following hold:
\begin{enumerate}
\item
\label{l:5498i-}
$Z$ is a closed convex subset of 
$\mathscr{P}\times\mathscr{D}$.
\item
\label{l:5498i}
$M$ is maximally monotone.
\item
\label{l:5498ii}
Suppose that $(x,y,v^*)\in\zer M$.
Then $(x,v^*)\in Z$, $x\in\mathscr{P}$,
and $v^*\in\mathscr{D}$.
\item
\label{l:5498iii}
$\mathscr{P}\neq\emp$ $\Leftrightarrow$
$\mathscr{D}\neq\emp$ $\Leftrightarrow$
$Z\neq\emp$ $\Leftrightarrow$
$\zer M\neq\emp$.
\end{enumerate}
\end{lemma}
\begin{proof}
\ref{l:5498i-}:
\cite[Proposition~2.1(i)(a)]{Joca16}.

\ref{l:5498i}:
Define
\begin{equation}
\begin{cases}
C\colon\XX\to 2^{\XX^*}\colon
(x,y,v^*)\mapsto(-s^*+Ax)\times By\times\{r\}\\
S\colon\XX\to\XX^*\colon
(x,y,v^*)\mapsto(L^*v^*,-v^*,-Lx+y).
\end{cases}
\end{equation}
It follows from the maximal monotonicity of $A$ and $B$ that 
$C$ is maximally monotone. On the other hand, 
$S$ is linear and bounded, and 
\begin{equation}
\label{e:221}
\big(\forall(x,y,v^*)\in\HH\big)\quad
\Pair{(x,y,v^*)}{S(x,y,v^*)}=
\pair{x}{L^*v^*}-\pair{y}{v^*}+\pair{y-Lx}{v^*}=0. 
\end{equation}
Thus, we derive from \cite[Section~17]{Simo08} that 
$S$ is maximally monotone with 
$\dom S=\HH$. In turn, \cite[Theorem~24.1(a)]{Simo08}
asserts that $M=C+S$ is maximally monotone.

\ref{l:5498ii}:
We deduce from \eqref{e:1571} that $s^*\in Ax+L^*v^*$, 
$v^*\in By$, and $y=Lx-r$; hence $v^*\in B(Lx-r)$.
Consequently, $s^*-L^*v^*\in Ax$ and $Lx-r\in B^{-1}v^*$,
which yields $(x,v^*)\in Z$. Finally, \ref{l:5498i-}
entails that $x\in\mathscr{P}$ and $v^*\in\mathscr{D}$.

\ref{l:5498iii}:
By \cite[Proposition~2.1(i)(c)]{Joca16},
$\mathscr{P}\neq\emp$ $\Leftrightarrow$
$\mathscr{D}\neq\emp$ $\Leftrightarrow$
$Z\neq\emp$. In addition, in view of \ref{l:5498ii},
$\zer M\neq\emp$ $\Rightarrow$ $Z\neq\emp$.
Suppose that $(x,v^*)\in Z$ and set $y=Lx-r$. Then 
$y=Lx-r\in B^{-1}v^*$ and $s^*\in Ax+L^*v^*$.
Hence $0\in By-v^*$ and $0\in -s^*+Ax+L^*v^*$.
Altogether, $0\in(-s^*+Ax+L^*v^*)\times(By-v^*)
\times\{r-Lx+y\}=M(x,y,v^*)$, 
i.e., $(x,y,v^*)\in\zer M$.
\end{proof}

Now suppose that $\XX$ is a real Hilbert space with scalar product
$\scal{\cdot}{\cdot}$. An operator $T\colon\HH\to\HH$
is nonexpansive if it is $1$-Lipschitzian,
$\alpha$-averaged with $\alpha\in\zeroun$ if
$\Id+(1/\alpha)(T-\Id)$ is nonexpansive, firmly nonexpansive if it
is $1/2$-averaged, and $\beta$-cocoercive with
$\beta\in\RPP$ if $\beta T$ is firmly nonexpansive. 
Averaged operators were introduced in \cite{Bail78}.
The projection operator onto a nonempty closed 
convex subset $C$ of $\HH$ is denoted by $\proj_C$. 
The resolvent of $M\colon\XX\to 2^{\XX}$ is $J_M=(\Id+M)^{-1}$. 

\section{Warped resolvents}
\label{sec:3}
We provide illustrations of Definition~\ref{d:wr} and then
study the properties of warped resolvents.

Our first example is the warped resolvent of a subdifferential. 
This leads to the following notion, which extends Moreau's
classical proximity operator in Hilbert spaces \cite{Mor62b}.

\begin{example}[Warped proximity operator]
\label{d:wpx}
Let $D$ be a nonempty subset of $\HH$, let $K\colon D\to\HH^*$, 
and let $\varphi\colon\HH\to\RX$ be a proper lower semicontinuous
convex function such that 
$\ran K\subset\ran(K+\partial\varphi)$ and $K+\partial\varphi$ 
is injective. The warped proximity operator of $\varphi$ with
kernel $K$ is $\prox_\varphi^K=(K+\partial\varphi)^{-1}\circ K$.
It is characterized by the variational inequality
\begin{equation}
\label{e:wmj}
\big(\forall (x,p)\in\HH\times\HH\big)\quad
p=\prox_\varphi^Kx\;\;\Leftrightarrow\;\;
(\forall y\in\HH)\quad\pair{y-p}{Kx-Kp}+\varphi(p)\leq\varphi(y).
\end{equation}
\end{example}

In particular, in the case of normal cones, we arrive at 
the following definition (see Figure~\ref{fig:wpj}).

\begin{example}[Warped projection operator]
\label{ex:wpj}
Let $D$ be a nonempty subset of $\HH$, let $K\colon D\to\HH^*$, and
let $C$ be a nonempty closed convex subset of $\HH$ with normal
cone operator $N_C$ such that 
$\ran K\subset\ran(K+N_C)$ and $K+N_C$ is injective. 
The warped projection operator onto $C$ with kernel $K$ is 
$\proj_C^K=(K+N_C)^{-1}\circ K$. It is characterized by 
\begin{equation}
\label{e:wpj}
\big(\forall (x,p)\in\HH\times\HH\big)\quad
p=\proj_C^Kx\;\;\Leftrightarrow\;\;
\big[\,p\in C\quad\text{and}\;\;
(\forall y\in C)\quad\pair{y-p}{Kx-Kp}\leq 0\,\big].
\end{equation}
\end{example}

\begin{figure}[ht!]
\scalebox{1.3} 
{
\begin{pspicture}(-6.5,-2.8)(6.0,3.0)
\psline[linewidth=0.03cm,linecolor=dgreen,arrowsize=0.04cm 4.0,%
arrowlength=1.4,arrowinset=0.4]{->}(0.0,2.75)(0,1.05)
\psline[linewidth=0.03cm,linecolor=dgreen,arrowsize=0.04cm 4.0,%
arrowlength=1.4,arrowinset=0.4]{->}(-2.5,-2.4)(-0.737,-0.737)
\psline[linewidth=0.03cm,linecolor=dgreen,arrowsize=0.04cm 4.0,%
arrowlength=1.4,arrowinset=0.4]{->}(2.5,-2.5)(0.737,-0.737)
\psline[linewidth=0.03cm,linecolor=red,arrowsize=0.04cm 4.0,%
arrowlength=1.4,arrowinset=0.4]{->}(0.707,-2.7)(0.707,-0.757)
\pscircle[linewidth=0.03cm,fillstyle=solid,fillcolor=color2b,%
linecolor=color2b](0,0){1.00}
\rput(0.0,0.8){\tiny $p_1$}
\rput(-0.53,-0.53){\tiny $p_2$}
\rput(0.53,-0.53){\tiny $p_3$}
\rput(0.0,1.0){\color{blue}\footnotesize $\bullet$}
\rput(0.707,-0.707){\color{blue}\footnotesize $\bullet$}
\rput(-0.707,-0.707){\color{blue}\footnotesize $\bullet$}
\psplot[plotpoints=800,algebraic,arrows=<-,arrowsize=0.15cm,%
linewidth=0.03cm,linecolor=red]{0.07}{1.43}{0.5*x^3+0.2*x+1}
\psplot[plotpoints=800,algebraic,arrows=<-,arrowsize=0.15cm,%
linewidth=0.03cm,linecolor=red]{-0.767}{-2.2}%
{0.25*x^3-0.4*x-51*1.4142/80}
\end{pspicture}
}
\caption{Warped projections onto the closed unit ball $C$ centered 
at the origin in the Euclidean plane. Sets of points projecting onto 
$p_1$, $p_2$, and $p_3$ for the kernels $K_1=\Id$ (in green) and 
$K_2\colon(\xi_1,\xi_2)\mapsto(\xi_1^3/2+\xi_1/5-\xi_2,
\xi_1+\xi_2)$ (in red). Note that $K_2$ is not a gradient.
}
\label{fig:wpj}
\end{figure}
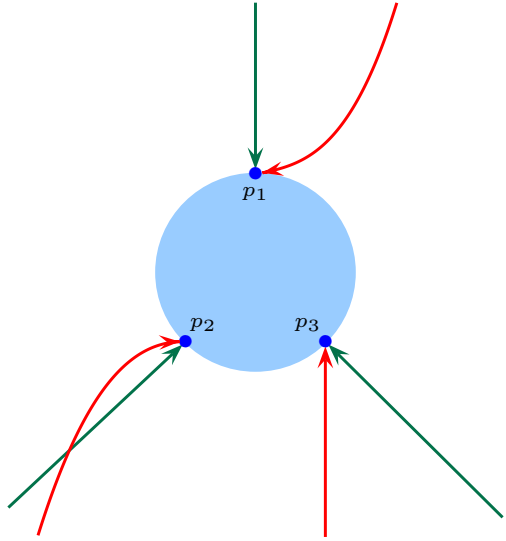 

\begin{example}
Suppose that $\XX$ is strictly convex,
let $M\colon\XX\to 2^{\XX^*}$ be maximally monotone, and let 
$K$ be the normalized duality mapping of $\XX$. Then $J_M^K$ is
a well-defined warped resolvent which was introduced in
\cite{Kass85}.
\end{example}

\begin{example}
Let $M\colon\XX\to 2^{\XX^*}$ be maximally monotone such that 
$\zer M\neq\emp$, let $f\colon\XX\to\RX$ be a Legendre function
\cite{Sico03} such that $\dom M\subset\intdom f$, and set 
$K=\nabla f$. Then it follows from
\cite[Corollary~3.14(ii)]{Sico03} that $J_M^K$ is a well-defined
warped resolvent, called the $D$-resolvent of $M$ in
\cite{Sico03}.
\end{example}

\begin{example}
Let $M\colon\XX\to 2^{\XX^*}$ be maximally monotone and let
$K\colon\XX\to\XX^*$ be strictly monotone, surjective, and
$3^*$ monotone in the sense that 
\cite[Definition~32.40(c)]{Zei90B}
\begin{equation}
(\forall x\in\dom M)(\forall x^*\in\ran M)\quad
\sup_{(y,y^*)\in\gra M}\pair{x-y}{y^*-x^*}<\pinf.
\end{equation}
Then it follows from \cite[Theorem~2.3]{Baus10}
that $J_M^K$ is a well-defined warped resolvent, called 
the $K$-resolvent of $M$ in \cite{Baus10}.
\end{example}

\begin{example}
Let $A\colon\XX\to 2^{\XX^*}$ and $B\colon\XX\to 2^{\XX^*}$
be maximally monotone, and let $f\colon\XX\to\RX$
be a proper lower semicontinuous convex function which is
essentially smooth \cite{Sico03}.
Suppose that $D=(\intdom f)\cap\dom A$ is a nonempty subset
of $\intdom B$, that $B$ is single-valued on $\intdom B$,
that $\nabla f$ is strictly monotone on $D$,
and that $(\nabla f-B)(D)\subset\ran(\nabla f+A)$.
Set $M=A+B$ and
$K\colon D\to\XX^*\colon x\mapsto\nabla f(x)-Bx$.
Then the warped resolvent $J_{M}^K$ is well defined and coincides
with the Bregman forward-backward operator $(\nabla
f+A)^{-1}\circ(\nabla f-B)$ investigated in \cite{Pape19}, where it
is shown to capture a construction found in \cite{Rena97}.
\end{example}

\begin{example}
Consider the setting of Lemma~\ref{l:5498}. For simplicity
(more general kernels can be considered), take
$s^*=0$, $r=0$, and assume that $\YY$ and $\ZZ^*$ are strictly 
convex, with normalized duality mapping $K_\YY$ and $K_{\ZZ^*}$. 
As seen in Lemma~\ref{l:5498}\ref{l:5498i-}, 
finding a zero of the Kuhn--Tucker operator
$U\colon\YY\times\ZZ^*\to 2^{\YY^*\times\ZZ}
\colon(x,v^*)\mapsto(Ax+L^*v^*)\times(B^{-1}v^*-Lx)$ provides a
solution to the primal-dual problem \eqref{e:p4}--\eqref{e:d4}.
Now set $K\colon(x,v^*)\mapsto(K_\YY x-L^*v^*,Lx+K_{\ZZ^*}v^*)$.
Then the warped resolvent $J_{U}^K$ is well defined and 
\begin{equation}
J_{U}^K\colon(x,v^*)\mapsto\big((K_\YY+A)^{-1}(K_\YY x-L^*v^*),
(K_{\ZZ^*}+B^{-1})^{-1}(Lx+K_{\ZZ^*}v^*)\big).
\end{equation}
For instance, in a Hilbertian setting, 
$J_{U}^K\colon(x,v^*)\mapsto(J_A(x-L^*v^*),J_{B^{-1}}(Lx+v^*))$, 
whereas $J_U$ is intractable; note also that the kernel 
$K$ is a non-Hermitian bounded linear operator.
\end{example}

Further examples will appear in Section~\ref{sec:5}. Let us turn
our attention to the properties of warped resolvents.

\begin{proposition}[viability]
\label{p:19}
Let $D$ be a nonempty subset of $\HH$, let $K\colon D\to\HH^*$, 
and let $M\colon\HH\to 2^{\HH^*}$ be such that 
$\ran K\subset\ran(K+M)$ and $K+M$ is injective. 
Then $J_M^K\colon D\to D$. 
\end{proposition}
\begin{proof}
By assumption, $\dom J_M^K=\dom((K+M)^{-1}\circ K)
=\menge{x\in\dom K}{Kx\in\dom(K+M)^{-1}}
=\menge{x\in D}{Kx\in\ran(K+M)}=D$. Next, observe that
\begin{equation}
\label{e:v}
\ran J_M^K=\ran\big((K+M)^{-1}\circ K\big)
\subset\ran(K+M)^{-1}=\dom (K+M)\subset\dom K=D.
\end{equation}
Finally, to show that $(K+M)^{-1}$ is at most
single-valued, suppose that $(x^*,x_1)\in\gra(K+M)^{-1}$
and $(x^*,x_2)\in\gra(K+M)^{-1}$.
Then $\{x^*\}\subset (K+M)x_1\cap(K+M)x_2$ and,
since $K+M$ is injective, it follows that $x_1=x_2$.
\end{proof}

Sufficient conditions that guarantee that warped resolvents are
well defined are made explicit below.

\begin{proposition}
\label{p:20}
Let $D$ be a nonempty subset of $\HH$, let $K\colon D\to\HH^*$, 
and let $M\colon\HH\to 2^{\HH^*}$. Then the following hold:
\begin{enumerate}
\item
\label{p:20ii}
Suppose that one of the following is satisfied:
\begin{enumerate}[label={\rm[\alph*]}]
\item
\label{p:20iia-}
$K+M$ is surjective.
\item
\label{p:20iia}
$K+M$ is maximally monotone and 
$D\cap\dom M$ is bounded.
\item
\label{p:20iib}
$K+M$ is maximally monotone,
$K+M$ is uniformly monotone with modulus $\phi$,
and $\phi(t)/t\to\pinf$ as $t\to\pinf$.
\item
\label{p:20iic}
$K+M$ is maximally monotone and 
strongly monotone.
\item
\label{p:20iie}
$M$ is maximally monotone, $D=\HH$, and $K$ is maximally monotone, 
strictly monotone, $3^*$ monotone, and surjective.
\item
\label{p:20iid}
$K$ is maximally monotone and there exists a lower semicontinuous 
coercive convex function $\varphi\colon\HH\to\RR$
such that $M=\partial\varphi$.
\end{enumerate}
Then $\ran K\subset\ran(K+M)$.
\item
\label{p:20i}
Suppose that one of the following is satisfied:
\begin{enumerate}[label={\rm[\alph*]}]
\item
\label{p:20ia}
$K+M$ is strictly monotone.
\item
\label{p:20ib}
$M$ is monotone and $K$ is strictly monotone on $\dom M$. 
\item
\label{p:20ic}
$K$ is monotone and $M$ is strictly monotone.
\item
\label{p:20id}
$-(K+M)$ is strictly monotone.
\end{enumerate}
Then $K+M$ is injective.
\end{enumerate}
\end{proposition}
\begin{proof}
Set $A=K+M$.

\ref{p:20ii}: Item \ref{p:20iia-} is clear. 
We prove the remaining ones as follows.

\ref{p:20iia}:
It follows from \cite[Theorem~32.G]{Zei90B} 
that $\ran A=\HH\supset\ran K$.

\ref{p:20iib}\&\ref{p:20iic}:
Since \cite[Lemma~2.7(ii)]{Joca16} and
\cite[Corollary~32.35]{Zei90B} assert that $A$ is
surjective, the claim follows from \ref{p:20ii}\ref{p:20iia-}.

\ref{p:20iie}: See \cite[Theorem~2.3]{Baus10}.

\ref{p:20iid}:
Take $z\in D$ and set $B=A(\mute+z)-Kz$.
By coercivity of $\varphi$, there exists $\rho\in\RPP$ such that
\begin{equation}
\label{e:0055}
(\forall x\in\HH)\quad
\|x\|\geq\rho
\quad\Rightarrow\quad
\inf\pair{x}{\partial\varphi(x+z)}
\geq\varphi(x+z)-\varphi(z)\geq 0.
\end{equation}
Now take $(x,x^*)\in\gra B$ and suppose that $\|x\|\geq\rho$.
Then $x^*+Kz-K(x+z)\in\partial\varphi(x+z)$ and it follows from
\eqref{e:0055} and the monotonicity of $K$ that
\begin{equation}
0\leq\pair{x}{x^*+Kz-K(x+z)}
=\pair{x}{x^*}-\pair{(x+z)-z}{K(x+z)-Kz}
\leq\pair{x}{x^*}.
\end{equation}
On the other hand, since $\dom\partial\varphi=\HH$ 
\cite[Theorems~2.2.20(b) and 2.4.12]{Zali02}, 
$A$ is maximally monotone \cite[Theorem~24.1(a)]{Simo08}, and
so is $B$. Altogether, \cite[Proposition~2]{Rock70} asserts that
there exists $\overline{x}\in\HH$ such that
$0\in B\overline{x}$. Consequently,
$Kz\in A(\overline{x}+z)\subset\ran(K+M)$.

\ref{p:20i}: We need to prove only \ref{p:20ia} since 
\ref{p:20ib} and \ref{p:20ic} are special cases of it, and 
\ref{p:20id} is similar. To this end, 
let $(x_1,x_2)\in\HH^2$ and suppose that $Ax_1\cap Ax_2\neq\emp$.
We must show that $x_1=x_2$. 
Take $x^*\in Ax_1\cap Ax_2$.
Then $(x_1,x^*)$ and $(x_2,x^*)$ lie in $\gra A$.
In turn, since $A$ is strictly monotone and
$\pair{x_1-x_2}{x^*-x^*}=0$, we obtain $x_1=x_2$.
\end{proof}

\begin{proposition}
\label{p:21}
Let $M\colon\HH\to 2^{\HH^*}$, let $\gamma\in\RPP$, and let
$K\colon\HH\to\HH^*$ be such that $\ran K\subset\ran(K+\gamma M)$ and
$K+\gamma M$ is injective. Then the following hold:
\begin{enumerate}
\item
\label{p:21ii}
$\Fix J_{\gamma M}^K=\zer M$.
\item
\label{p:21iii}
Let $x\in\HH$ and $p\in\HH$. Then $p=J_{\gamma M}^Kx$ 
$\Leftrightarrow$ $(p,\gamma^{-1}(Kx-Kp))\in\gra M$.
\item
\label{p:21iv}
Suppose that $M$ is monotone.
Let $x\in\HH$ and $y\in\HH$, and set
$p=J_{\gamma M}^Kx$ and 
$q=J_{\gamma M}^Ky$. Then
$\pair{p-q}{Kx-Ky}\geq\pair{p-q}{Kp-Kq}$.
\item
\label{p:21vi}
Suppose that $M$ is monotone, 
that $K$ is uniformly continuous and $\phi$-uniformly 
monotone, and that $\psi\colon t\mapsto\phi(t)/t$
is real-valued on $\left]0,\xi\right[$ for some
$\xi\in\RPP$ and strictly increasing.
Then $J_{\gamma M}^K$ is uniformly continuous.
\item
\label{p:21vii}
Suppose that $M$ is monotone and that $K$ is
$\beta$-Lipschitzian and $\alpha$-strongly monotone
for some $\alpha\in\RPP$ and $\beta\in\RPP$.
Then $J_{\gamma M}^K$ is $(\beta/\alpha)$-Lipschitzian.
\item
\label{p:21v}
Suppose that $M$ is monotone.
Let $x\in\HH$, and set $y=J_{\gamma M}^Kx$ and 
$y^*=\gamma^{-1}(Kx-Ky)$. Then
$\zer M\subset\menge{z\in\HH}{\pair{z-y}{y^*}\leq 0}$.
\end{enumerate}
\end{proposition}
\begin{proof}
\ref{p:21ii}:
We derive from Proposition~\ref{p:19} that
$(\forall x\in\HH)$
$x\in\zer M$ $\Leftrightarrow$ $Kx\in Kx+\gamma Mx$
$\Leftrightarrow$ $x=J^K_{\gamma M}x$
$\Leftrightarrow$ $x\in\Fix J^K_{\gamma M}$.

\ref{p:21iii}:
We have $p=J^K_{\gamma M}x$
$\Leftrightarrow$ $p=(K+\gamma M)^{-1}(Kx)$
$\Leftrightarrow$ $Kx\in Kp+\gamma Mp$
$\Leftrightarrow$ $Kx-Kp\in\gamma Mp$
$\Leftrightarrow$ $(p,\gamma^{-1}(Kx-Kp))\in\gra M$.

\ref{p:21iv}: This follows from \ref{p:21iii} and the monotonicity
of $M$.

\ref{p:21vi}: Let $x$ and $y$ be in $\HH$, and set
$p=J_{\gamma M}^Kx$ and $q=J_{\gamma M}^Ky$.
Then we deduce from \ref{p:21iv} that
\begin{equation}
\label{e:9u7} 
\phi(\|p-q\|)\leq\pair{p-q}{Kp-Kq}
\leq\pair{p-q}{Kx-Ky}\leq\|p-q\|\,\|Kx-Ky\|.
\end{equation}
Now fix $\varepsilon\in\left]0,\xi\right[$ and let 
$\eta\in\left]0,\psi(\varepsilon)\right]$.
By uniform continuity of $K$, there exists 
$\delta\in\RPP$ such that 
$\|x-y\|\leq\delta$ $\Rightarrow$ $\|Kx-Ky\|\leq\eta$.
Without loss of generality, suppose that $p\neq q$.
Then, if $\|x-y\|\leq\delta$, we derive from \eqref{e:9u7} that
$\psi(\|p-q\|)\leq\|Kx-Ky\|\leq\eta\leq\psi(\varepsilon)$.
Consequently, since $\psi$ is strictly increasing,
$\|p-q\|\leq\varepsilon$.

\ref{p:21vii}: Let $x$ and $y$ be 
in $\HH$ and set
$p=J_{\gamma M}^Kx$ and $q=J_{\gamma M}^Ky$. Then
we deduce from \ref{p:21iv} that
\begin{equation}
\label{e:9u8} 
\alpha\|p-q\|^2\leq\pair{p-q}{Kp-Kq}
\leq\pair{p-q}{Kx-Ky}\leq\|p-q\|\,\|Kx-Ky\|
\leq\beta\|p-q\|\,\|x-y\|.
\end{equation}
In turn, $\|p-q\|\leq(\beta/\alpha)\|x-y\|$.

\ref{p:21v}: Suppose that $z\in\zer M$. Then $(z,0)\in\gra M$. On
the other hand, we derive from \ref{p:21iii} that
$(y,y^*)\in\gra M$. Hence, by monotonicity of $M$, 
$\pair{y-z}{y^*}\geq 0$.
\end{proof}

In Hilbert spaces, standard resolvents are firmly nonexpansive,
hence $1/2$-averaged. A related property for warped resolvents is
the following.

\begin{proposition}
\label{p:22}
Suppose that $\XX$ is a Hilbert space.
Let $M\colon\HH\to 2^{\HH}$ be maximally monotone
and let $K\colon\HH\to\HH$ be averaged with constant
$\alpha\in\zeroun$. Suppose that $K+M$ is $1$-strongly monotone.
Then $J^K_M$ is averaged with constant $1/(2-\alpha)$.
\end{proposition}
\begin{proof}
Since $K$ is nonexpansive by virtue of
\cite[Remark~4.34(i)]{Livre1},
it follows from the Cauchy--Schwarz inequality that
\begin{align}
\label{e:3251}
(\forall x\in\HH)(\forall y\in\HH)\quad
\scal{x-y}{(2\Id+K)x-(2\Id+K)y}
&=2\|x-y\|^2+\scal{x-y}{Kx-Ky}
\nonumber\\
&\geq 2\|x-y\|^2-\|x-y\|^2
\nonumber\\
&=\|x-y\|^2
\end{align}
and therefore, by continuity of $2\Id+K$,
that $2\Id+K$ is maximally monotone \cite[Corollary~20.28]{Livre1}.
Thus, in the light of \cite[Corollary~25.5(i)]{Livre1},
$2\Id+K+M$ is maximally monotone.
In turn, since $2\Id+K+M$ is strongly monotone by \eqref{e:3251},
\cite[Proposition~22.11(ii)]{Livre1} entails that
$\ran(3\Id+K+M-\Id)=\ran(2\Id+K+M)=\HH$,
which yields $\ran(\Id+(K+M-\Id)/3)=\HH$.
Hence, by monotonicity of $K+M-\Id$ and
Minty's theorem \cite[Theorem~21.1]{Livre1},
we infer that $K+M-\Id$ is maximally monotone.
Thus, in view of \cite[Corollary~23.9]{Livre1},
$(K+M)^{-1}=(\Id+K+M-\Id)^{-1}$ is averaged with constant
$1/2$. Consequently, we infer from \cite[Proposition~4.44]{Livre1}
that $J^K_M=(K+M)^{-1}\circ K$ is averaged with constant
$1/(2-\alpha)$.
\end{proof}

\section{Warped proximal iterations}
\label{sec:4}
Throughout this section, $\HH$ is a real Hilbert space identified
with its dual. We start with an abstract principle for the basic 
problem of finding a zero of a maximally monotone operator. 

\begin{proposition}
\label{p:1}
Let $M\colon\HH\to 2^{\HH}$ be a maximally monotone operator
such that $Z=\zer M\neq\emp$, let $x_0\in\HH$, let 
$\varepsilon\in\zeroun$, let $(\lambda_n)_{n\in\NN}$ be a sequence
in $[\varepsilon,2-\varepsilon]$, and let 
$(y_n,y_n^*)_{n\in\NN}$ be a sequence in $\gra M$. Set
\begin{equation}
\label{e:fejer14}
(\forall n\in\NN)\quad
x_{n+1}=
\begin{cases}
x_n+\dfrac{\lambda_n\scal{y_n-x_n}{y_n^*}}
{\|y_n^*\|^2}\,y_n^*,&\text{if}\:\:\scal{y_n-x_n}{y_n^*}<0;\\
x_n,&\text{otherwise.}
\end{cases}
\end{equation}
Then the following hold:
\begin{enumerate}
\item
\label{p:1i}
$\sum_{n\in\NN}\|x_{n+1}-x_n\|^2<\pinf$.
\item
\label{p:1ii}
Suppose that every weak sequential cluster point of
$(x_n)_{n\in\NN}$ is in $Z$. Then $(x_n)_{n\in\NN}$ converges
weakly to a point in $Z$.
\end{enumerate}
\end{proposition}
\begin{proof}
By \cite[Proposition~23.39]{Livre1}, $Z$ is a
nonempty closed convex subset of $\XX$. Set $(\forall n\in\NN)$ 
$H_n=\menge{z\in\HH}{\scal{z-y_n}{y_n^*}\leq 0}$.
For every $z\in Z$ and every $n\in\NN$,
since $(z,0)$ and $(y_n,y_n^*)$ lie in $\gra M$,
the monotonicity of $M$ forces $\scal{y_n-z}{y_n^*}\geq 0$.
Thus $Z\subset\bigcap_{n\in\NN}H_n$.
In addition, \cite[Example~29.20]{Livre1} asserts that
\begin{equation}
(\forall n\in\NN)\quad\proj_{H_n}x_n=
\begin{cases}
x_n+\dfrac{\scal{y_n-x_n}{y_n^*}}{\|y_n^*\|^2}\,y_n^*,&
\text{if}\;\scal{y_n-x_n}{y_n^*}<0;\\
x_n,&\text{otherwise}.
\end{cases}
\end{equation}
Hence, we derive from \eqref{e:fejer14} that
\begin{equation}
\label{e:fejer13}
(\forall n\in\NN)\quad
x_{n+1}=x_n+\lambda_n(\proj_{H_n}x_n-x_n).
\end{equation}
Therefore \ref{p:1i} follows from \cite[Equation~(10)]{Eoop01} and
\ref{p:1ii} follows from \cite[Proposition~6i)]{Eoop01}.
\end{proof}

To implement the conceptual principle outlined in 
Proposition~\ref{p:1}, one is required to construct points 
in the graph of the underlying monotone operator.
Towards this end,
our strategy is to use Proposition~\ref{p:21}\ref{p:21iii}.
We shall then seamlessly obtain in Section~\ref{sec:5}
a broad class of algorithms to solve a variety of
monotone inclusions. It will be convenient to use the notation
\begin{equation}
\label{e:1826}
(\forall y^*\in\GG^*)\quad\norma{(y^*)}=
\begin{cases}
\dfrac{y^*}{\|y^*\|},&\text{if}\;y^*\neq 0;\\
0,&\text{if}\;y^*=0.
\end{cases}
\end{equation}

Our first method employs, at iteration $n$, a warped resolvent 
based on a different kernel, and this warped resolvent
is applied at a point $\widetilde{x}_n$ that may not be the current
iterate $x_n$.

\begin{theorem}
\label{t:1}
Let $M\colon\HH\to 2^{\HH}$ be a maximally monotone
operator such that $Z=\zer M\neq\emp$, let $x_0\in\HH$,
let $\varepsilon\in\zeroun$, let $(\lambda_n)_{n\in\NN}$ be a
sequence in $[\varepsilon,2-\varepsilon]$, and
let $(\gamma_n)_{n\in\NN}$ be a sequence in 
$\left[\varepsilon,\pinf\right[$. Further,
for every $n\in\NN$, let $\widetilde{x}_n\in\HH$ and let 
$K_n\colon\HH\to\HH$ be a monotone
operator such that $\ran K_n\subset\ran(K_n+\gamma_nM)$
and $K_n+\gamma_nM$ is injective. Iterate
\begin{equation}
\label{e:fejer15}
\begin{array}{l}
\text{for}\;n=0,1,\ldots\\
\left\lfloor
\begin{array}{l}
y_n=J_{\gamma_n M}^{K_n}\widetilde{x}_n\\
y_n^*=\gamma_n^{-1}(K_n\widetilde{x}_n-K_ny_n)\\
\text{if}\;\scal{y_n-x_n}{y_n^*}<0\\
\left\lfloor
\begin{array}{l}
x_{n+1}=x_n+\dfrac{\lambda_n\sscal{y_n-x_n}{y_n^*}}
{\|y_n^*\|^2}\,y_n^*\\
\end{array}
\right.\\
\text{else}\\
\left\lfloor
\begin{array}{l}
x_{n+1}=x_n.
\end{array}
\right.\\[2mm]
\end{array}
\right.\\
\end{array}
\end{equation}
Then the following hold:
\begin{enumerate}
\item
\label{t:1i}
$\sum_{n\in\NN}\|x_{n+1}-x_n\|^2<\pinf$.
\item
\label{t:1ii}
Suppose that the following are satisfied:
\begin{enumerate}[label={\rm[\alph*]}]
\item
\label{t:1iia}
$\widetilde{x}_n-x_n\to 0$.
\item
\label{t:1iib}
$\sscal{\widetilde{x}_n-y_n}{\norma{(K_n\widetilde{x}_n-K_ny_n)}}
\to 0\quad\Rightarrow\quad
\begin{cases}
\widetilde{x}_n-y_n\weakly 0\\
K_n\widetilde{x}_n-K_ny_n\to 0.
\end{cases}$
\end{enumerate}
Then $(x_n)_{n\in\NN}$ converges weakly to a point in $Z$.
\end{enumerate}
\end{theorem}
\begin{proof}
By Proposition~\ref{p:21}\ref{p:21iii},
\begin{equation}
\label{e:3817}
(\forall n\in\NN)\quad(y_n,y_n^*)\in\gra M.
\end{equation}
Therefore, \ref{t:1i} follows from Proposition~\ref{p:1}\ref{p:1i}.
It remains to prove \ref{t:1ii}. To this end, take a strictly
increasing sequence $(k_n)_{n\in\NN}$ in $\NN$
and a point $x\in\HH$ such that $x_{k_n}\weakly x$. 
In view of Proposition~\ref{p:1}\ref{p:1ii},
we must show that $x\in Z$. We infer from
\ref{t:1ii}\ref{t:1iia} that
\begin{equation}
\label{e:8317}
\widetilde{x}_{k_n}\weakly x.
\end{equation}
Next,
by \eqref{e:1826} and \eqref{e:fejer15}, for every $n\in\NN$,
if $\scal{x_n-y_n}{y_n^*}>0$, then
$y_n^*\neq 0$ and 
\begin{equation}
\sscal{x_n-y_n}{\norma{(y_n^*)}}
=\dfrac{\scal{x_n-y_n}{y_n^*}}{\|y_n^*\|}
=\lambda_n^{-1}\|x_{n+1}-x_n\|
\leq\varepsilon^{-1}\|x_{n+1}-x_n\|;
\end{equation}
otherwise, $\scal{x_n-y_n}{y_n^*}\leq 0$ and
it thus results from \eqref{e:1826} that
\begin{align}
\sscal{x_n-y_n}{\norma{(y_n^*)}}
&=
\begin{cases}
0,&\text{if}\;y_n^*=0;\\
\dfrac{\scal{x_n-y_n}{y_n^*}}{\|y_n^*\|},
&\text{otherwise}
\end{cases}
\nonumber\\
&\leq 0
\nonumber\\
&=\varepsilon^{-1}\|x_{n+1}-x_n\|.
\end{align}
Therefore, using \ref{t:1i}
and the monotonicity of $(K_n)_{n\in\NN}$, we obtain
\begin{align}
\label{e:93485t}
0&\leftarrow\varepsilon^{-1}\|x_{n+1}-x_n\|\nonumber\\
&\geq\scal{x_n-y_n}{\norma{(y_n^*)}}\nonumber\\
&=\sscal{x_n-\widetilde{x}_n}{\norma{(K_n\widetilde{x}_n-K_ny_n)}}
+\sscal{\widetilde{x}_n-y_n}{\norma{(K_n\widetilde{x}_n-K_ny_n)}}
\nonumber\\
&\geq\sscal{x_n-\widetilde{x}_n}
{\norma{(K_n\widetilde{x}_n-K_ny_n)}}.
\end{align}
However, by the Cauchy--Schwarz inequality
and \ref{t:1ii}\ref{t:1iia}, 
\begin{equation}
\abscal{x_n-\widetilde{x}_n}
{\norma{(K_n\widetilde{x}_n-K_ny_n)}}
\leq\|x_n-\widetilde{x}_n\|\to 0.
\end{equation}
Hence, \eqref{e:93485t} implies that
$\scal{\widetilde{x}_n-y_n}
{\norma{(K_n\widetilde{x}_n-K_ny_n)}}\to 0$.
In turn, we deduce from \ref{t:1ii}\ref{t:1iib} that
$\widetilde{x}_n-y_n\weakly 0$ and
$K_n\widetilde{x}_n-K_ny_n\to 0$. Altogether,
since $\sup_{n\in\NN}\gamma_n^{-1}\leq\varepsilon^{-1}$,
it follows from \eqref{e:3817} and \eqref{e:8317} that
\begin{equation}
y_{k_n}=\widetilde{x}_{k_n}+(y_{k_n}-\widetilde{x}_{k_n})
\weakly x
\end{equation}
and
\begin{equation}
My_{k_n}\ni y_{k_n}^*=
\gamma_{k_n}^{-1}(K_{k_n}\widetilde{x}_{k_n}-K_{k_n}y_{k_n})\to 0.
\end{equation}
Appealing to the maximal monotonicity of $M$, 
\cite[Proposition~20.38(ii)]{Livre1} allows us to conclude that 
$x\in Z$.
\end{proof}

\begin{remark}
\label{r:3}
Condition~\ref{t:1ii}\ref{t:1iib} in Theorem~\ref{t:1} is satisfied
in particular when there exist $\alpha$ and $\beta$ in $\RPP$ such
that the kernels $(K_n)_{n\in\NN}$ are $\alpha$-strongly monotone
and $\beta$-Lipschitzian.
\end{remark}

\begin{remark}
\label{r:2}
The auxiliary sequence $(\widetilde{x}_n)_{n\in\NN}$
in Theorem~\ref{t:1} can serve several purposes.
In general, it provides the flexibility
of not applying the warped resolvent to the current iterate.
Here are some noteworthy candidates.
\begin{enumerate}
\item
At iteration $n$, $\widetilde{x}_n$ can model an additive 
perturbation of $x_n$, say $\widetilde{x}_n=x_n+e_n$. Here the
error sequence $(e_n)_{n\in\NN}$ need only satisfy $\|e_n\|\to 0$
and not the usual summability condition
$\sum_{n\in\NN}\|e_n\|<\pinf$ required
in many methods, e.g., \cite{Bot13c,Siop13,Svva12,Bang13}.
\item
\label{r:2ii}
Mimicking the behavior of so-called inertial methods 
\cite{Atto19,Siop17}, 
let $(\alpha_n)_{n\in\NN}$ be a bounded sequence in $\RR$ and set
$(\forall n\in\NN\smallsetminus\{0\})$ 
$\widetilde{x}_n=x_n+\alpha_n(x_n-x_{n-1})$. Then
Theorem~\ref{t:1}\ref{t:1i}
yields $\|\widetilde{x}_n-x_n\|=|\alpha_n|\,\|x_n-x_{n-1}\|\to 0$
and therefore
assumption \ref{t:1ii}\ref{t:1iia} holds in
Theorem~\ref{t:1}. More generally, weak convergence
results can be derived from Theorem~\ref{t:1} for iterations with
memory, that is,
\begin{equation}
(\forall n\in\NN)\quad \widetilde{x}_n=\sum_{j=0}^n\mu_{n,j}x_j,
\quad\text{where}\quad
(\mu_{n,j})_{0\leq j\leq n}\in\RR^{n+1}
\quad\text{and}\quad
\sum_{j=0}^n\mu_{n,j}=1.
\end{equation}
Here condition \ref{t:1ii}\ref{t:1iia} holds if
$(1-\mu_{n,n})x_n-\sum_{j=0}^{n-1}\mu_{n,j}x_j\to 0$.
In the case of standard inertial methods, weak convergence requires
more stringent conditions on the weights $(\mu_{n,j})_{n\in\NN,
0\leq j\leq n}$ \cite{Siop17}.
\item
Nonlinear perturbations can also be considered. For instance, 
at iteration $n$, $\widetilde{x}_n=\proj_{C_n}x_n$ is an
approximation to $x_n$ from some suitable closed convex set
$C_n\subset\XX$. 
\end{enumerate}
\end{remark}

\begin{remark}
The independent work \cite{Gise19} was posted on arXiv at the same
time as the report \cite{Warp19} from which our paper is derived. 
The former uses a notion of resolvents subsumed by
Definition~\ref{d:wr} to explore
the application of an algorithm similar to \eqref{e:fejer15} with
no perturbation (i.e., for every $n\in\NN$,
$\widetilde{x}_n=x_n$). The work \cite{Gise19} nicely complements
ours in the sense that it proposes applications to
splitting schemes not discussed here, which further
attests to the versatility and effectiveness of the notion of 
warped proximal iterations. 
\end{remark}

We now turn our attention to a variant of Theorem~\ref{t:1} that
guarantees strong convergence of the iterates to a best
approximation. In the spirit of Haugazeau's algorithm 
(see \cite[Th\'eor\`eme~3-2]{Haug68} and 
\cite[Corollary~30.15]{Livre1}), it
involves outer approximations consisting of
the intersection of two half-spaces. 
For convenience, given $(x,y,z)\in\HH^3$, we set
\begin{equation}
\label{e:blah}
H(x,y)=\menge{u\in\HH}{\scal{u-y}{x-y}\leq 0}
\end{equation}
and, if $R=H(x,y)\cap H(y,z)\neq\emp$, $Q(x,y,z)=\proj_Rx$. The
latter can be computed explicitly as follows
(see \cite[Th\'eor\`eme~3-1]{Haug68} or 
\cite[Corollary~29.25]{Livre1}).

\begin{lemma}
\label{l:haugazeauy}
Let $(x,y,z)\in\HH^3$. Set $R=H(x,y)\cap H(y,z)$, 
$\chi=\scal{x-y}{y-z}$, $\mu=\|x-y\|^2$,
$\nu=\|y-z\|^2$, and $\rho=\mu\nu-\chi^2$.
Then exactly one of the following holds:
\begin{enumerate}
\item
\label{c:haugazeaui}
$\rho=0$ and $\chi<0$, in which case $R=\emp$.
\item 
\label{c:haugazeauii}
\emph{[}$\,\rho=0$ and $\chi\geq 0\,$\emph{]} or 
$\rho>0$, in which case ${R}\neq\emp$ and 
\begin{equation}
\label{e:01-05}
Q(x,y,z)=
\begin{cases}
z,&\text{if}\;\rho=0\;\text{and}\;
\chi\geq 0;\\[+0mm]
\displaystyle
x+(1+\chi/\nu)(z-y), 
&\text{if}\;\rho>0\;\text{and}\;
\chi\nu\geq\rho;\\
\displaystyle y+(\nu/\rho)
\big(\chi(x-y)+\mu(z-y)\big), 
&\text{if}\;\rho>0\;\text{and}\;\chi\nu<\rho.
\end{cases}
\end{equation}
\end{enumerate}
\end{lemma}

Our second abstract convergence principle can now be stated.

\begin{proposition}
\label{p:2}
Let $M\colon\HH\to 2^{\HH}$ be a maximally monotone operator
such that $Z=\zer M\neq\emp$, let $x_0\in\HH$, and let
$(y_n,y_n^*)_{n\in\NN}$ be a sequence in $\gra M$.
For every $n\in\NN$, set
\begin{equation}
\label{e:fejer36}
x_{n+1/2}=
\begin{cases}
x_n+\dfrac{\scal{y_n-x_n}{y_n^*}}{\|y_n^*\|^2}\,y_n^*,
&\text{if}\:\:\scal{y_n-x_n}{y_n^*}<0;\\
x_n,&\text{otherwise}
\end{cases}
\quad\text{and}\quad
x_{n+1}=Q\big(x_0,x_n,x_{n+1/2}\big).
\end{equation}
Then the following hold:
\begin{enumerate}
\item
\label{p:2i}
$\sum_{n\in\NN}\|x_{n+1}-x_n\|^2<\pinf$ and
$\sum_{n\in\NN}\|x_{n+1/2}-x_n\|^2<\pinf$.
\item
\label{p:2ii}
Suppose that every weak sequential cluster point of
$(x_n)_{n\in\NN}$ is in $Z$. Then $(x_n)_{n\in\NN}$ converges
strongly to $\proj_Zx_0$.
\end{enumerate}
\end{proposition}
\begin{proof}
Set $(\forall n\in\NN)$ 
$H_n=\menge{z\in\HH}{\scal{z-y_n}{y_n^*}\leq 0}$.
Then, as in the proof of Proposition~\ref{p:1},
$Z$ is a nonempty closed convex subset of $\XX$ and
$Z\subset\bigcap_{n\in\NN}H_n$. On the one hand,
\begin{equation}
\label{e:h1}
(\forall n\in\NN)\quad
x_{n+1/2}=\proj_{H_n}x_n\quad\text{and}\quad
x_{n+1}=Q\big(x_0,x_n,x_{n+1/2}\big).
\end{equation}
On the other hand, by \eqref{e:blah},
\begin{align}
(\forall n\in\NN)\quad H\big(x_n,x_{n+1/2}\big)&=
\begin{cases}
\HH,&\text{if}\;x\in H_n;\\
H_n,&\text{otherwise}
\end{cases}
\nonumber\\
&\supset Z.
\end{align}
The claims therefore follow from \cite[Proposition~2.1]{Nfao15}.
\end{proof}

\begin{theorem}
\label{t:2}
Let $M\colon\HH\to 2^{\HH}$ be a maximally monotone operator such
that $Z=\zer M\neq\emp$, let $x_0\in\HH$, 
and let $(\gamma_n)_{n\in\NN}$ be a sequence in $\RPP$
such that $\inf_{n\in\NN}\gamma_n>0$.
For every $n\in\NN$, let $\widetilde{x}_n\in\HH$ and let
$K_n\colon\HH\to\HH$ be a monotone operator such that
$\ran K_n\subset\ran(K_n+\gamma_nM)$
and $K_n+\gamma_nM$ is injective. Iterate
\begin{equation}
\label{e:fejer16}
\begin{array}{l}
\text{for}\;n=0,1,\ldots\\
\left\lfloor
\begin{array}{l}
y_n=J_{\gamma_n M}^{K_n}\widetilde{x}_n\\
y_n^*=\gamma_n^{-1}(K_n\widetilde{x}_n-K_ny_n)\\
\text{if}\;\scal{y_n-x_n}{y_n^*}< 0\\
\left\lfloor
\begin{array}{l}
x_{n+1/2}=x_n+\dfrac{\scal{y_n-x_n}{y_n^*}}{\|y_n^*\|^2}\,
y_n^*\\
\end{array}
\right.\\
\text{else}\\
\left\lfloor
\begin{array}{l}
x_{n+1/2}=x_n\\
\end{array}
\right.\\
x_{n+1}=Q(x_0,x_n,x_{n+1/2}).\\
\end{array}
\right.\\
\end{array}
\end{equation}
Then the following hold:
\begin{enumerate}
\item
\label{t:2i}
$\sum_{n\in\NN}\|x_{n+1}-x_n\|^2<\pinf$ and
$\sum_{n\in\NN}\|x_{n+1/2}-x_n\|^2<\pinf$.
\item
\label{t:2ii}
Suppose that the following are satisfied:
\begin{enumerate}[label={\rm[\alph*]}]
\item
\label{t:2iia}
$\widetilde{x}_n-x_n\to 0$.
\item
\label{t:2iib}
$\sscal{\widetilde{x}_n-y_n}{\norma{(K_n\widetilde{x}_n-K_ny_n)}}
\to 0\quad\Rightarrow\quad
\begin{cases}
\widetilde{x}_n-y_n\weakly 0\\
K_n\widetilde{x}_n-K_ny_n\to 0.
\end{cases}$
\end{enumerate}
Then $(x_n)_{n\in\NN}$ converges strongly to $\proj_Zx_0$.
\end{enumerate}
\end{theorem}
\begin{proof}
Proposition~\ref{p:21}\ref{p:21iii} asserts that
$(\forall n\in\NN)$ $(y_n,y_n^*)\in\gra M$.
Thus, we obtain \ref{t:2i} from Proposition~\ref{p:2}\ref{p:2i}.
In the light of Proposition~\ref{p:2}\ref{p:2ii},
to establish \ref{t:2ii}, we need to show that every weak
sequential cluster point of $(x_n)_{n\in\NN}$ is a zero of $M$.
Since \ref{t:2i} asserts that $x_{n+1/2}-x_n\to 0$, this is done
as in the proof of Theorem~\ref{t:1}\ref{t:1ii}.
\end{proof}

We complete this section with the following remarks.

\begin{remark}
\label{r:8}
Suppose that $\GG$ and $\KK$ are real Hilbert spaces
and that $\HH=\GG\times\KK$.
Let $A\colon\GG\to 2^{\GG}$ and $B\colon\KK\to 2^{\KK}$
be maximally monotone, and let $L\in\BL(\GG,\KK)$. Define
\begin{equation}
\label{e:9999}
M\colon\HH\to 2^{\HH}\colon
(x,v^*)\mapsto (Ax+L^*v^*)\times(-Lx+B^{-1}v^*).
\end{equation}
In \cite{Siop14,Nfao15,MaPr18} the problem of finding a zero of
$M$ (and hence a solution to the monotone inclusion
$0\in Ax+L^*(B(Lx))$) is approached by generating, at each
iteration $n$, points $(a_n,a_n^*)\in\gra A$ and
$(b_n,b_n^*)\in\gra B$. This does provide a point
$(y_n,y_n^*)=
((a_n,b_n^*),(a_n^*+L^*b_n^*,-La_n+b_n))\in\gra M$,
which shows that the algorithms proposed in 
\cite{Siop14,Nfao15,MaPr18} are actually instances of
the conceptual principles laid out in
Propositions~\ref{p:1} and \ref{p:2}.
In particular, the primal-dual framework of 
\cite{Siop14} corresponds to applying Theorem~\ref{t:1} to
the operator $M$ of \eqref{e:9999} with kernels
\begin{equation}
\label{e:k4}
(\forall n\in\NN)\quad K_n\colon\XX\to {\XX}
\colon(x,v^*)\mapsto\big(\gamma_n^{-1}x-L^*v^*,
Lx+\mu_nv^*\big).
\end{equation}
Likewise, that of \cite{Nfao15} corresponds to the application of
Theorem~\ref{t:2} to this setting.
\end{remark}

\begin{remark}
\label{r:1234}
In Theorems~\ref{t:1} and \ref{t:2}, the algorithms operate
by using a single point $(y_n,y_n^*)$ in $\gra M$
at iteration $n$. 
It may be advantageous to use a finite family
$(y_{i,n},y^*_{i,n})_{i\in I_n}$ of points in $\gra M$, say
\begin{equation}
(\forall i\in I_n)\quad
(y_{i,n},y^*_{i,n})=\Big(J_{\gamma_{i,n}M}^{K_{i,n}}
\widetilde{x}_{i,n},
\gamma_{i,n}^{-1}(K_{i,n}\widetilde{x}_{i,n}-K_{i,n}y_{i,n})\Big).
\end{equation}
By monotonicity of $M$,
$(\forall i\in I_n)(\forall z\in\zer M)$
$\scal{z}{y_{i,n}^*}\leq\scal{y_{i,n}}{y_{i,n}^*}$.
Therefore, using ideas found in the area of convex feasibility
algorithms \cite{Cras95,Lopu97}, at every iteration $n$,
given strictly positive weights 
$(\omega_{i,n})_{i\in I_n}$ adding up to $1$, we average
these inequalities to create a new half-space $H_n$ containing
$\zer M$, namely
\begin{equation}
\zer M\subset H_n=\menge{z\in\HH}{\scal{z}{y_n^*}\leq\eta_n},
\quad\text{where}\quad 
\begin{cases}
y_n^*=\sum_{i\in I_n}\omega_{i,n}y_{i,n}^*\\
\eta_n=\sum_{i\in I_n}\omega_{i,n}\scal{y_{i,n}}{y_{i,n}^*}.
\end{cases}
\end{equation}
Now set 
\begin{equation}
\Lambda_n=
\begin{cases}
\dfrac{\sum_{i\in I_n}\omega_{i,n}\scal{y_{i,n}-x_n}{y_{i,n}^*}}
{\big\|\sum_{i\in I_n}\omega_{i,n}y_{i,n}^*\big\|^2},
&\text{if}\;\sum_{i\in I_n}\omega_{i,n}
\scal{x_n-y_{i,n}}{y^*_{i,n}}>0;\\
0,&\text{otherwise.}
\end{cases}
\end{equation}
Then, employing $\proj_{H_n}x_n=x_n+\Lambda_n\sum_{i\in I_n}
\omega_{i,n}y^*_{i,n}$ as the point $x_{n+1}$ in \eqref{e:fejer15}
and as the point $x_{n+1/2}$ in \eqref{e:fejer16} results in
multi-point extensions of Theorems~\ref{t:1} and \ref{t:2}.
\end{remark}

\section{Applications}
\label{sec:5}
We apply Theorem~\ref{t:1} to design new algorithms to solve
complex monotone inclusion problems in a real Hilbert space $\HH$. 
We do not mention explicitly
minimization problems as they follow, with usual constraint
qualification conditions, by considering monotone inclusions
involving subdifferentials as maximally monotone operators
\cite{Livre1,Siop13}. For brevity, we do not mention either 
the strongly convergent counterparts of each of the 
corollaries below that can be systematically obtained 
using Theorem~\ref{t:2}.

Let us note that the most basic instantiation of Theorem~\ref{t:1}
is obtained by setting $(\forall n\in\NN)$ $K_n=\Id$,
$\widetilde{x}_n=x_n$, and $\lambda_n=1$. In this case, the warped
proximal algorithm \eqref{e:fejer15} reduces to the basic proximal
point algorithm \eqref{e:prox1}.

In connection with Remark~\ref{r:2},
let us first investigate the convergence of a novel perturbed 
forward-backward-forward algorithm with memory. This will require
the following fact.

\begin{lemma}
\label{l:3295}
Let $B\colon\HH\to\HH$ be Lipschitzian with constant 
$\beta\in\RPP$, let $W\colon\HH\to\HH$ be strongly monotone
with constant $\alpha\in\RPP$,
let $\varepsilon\in\left]0,\alpha\right[$,
let $\gamma\in\left]0,(\alpha-\varepsilon)/\beta\right]$,
and set $K=W-\gamma B$. Then the following hold:
\begin{enumerate}
\item
\label{l:3295i}
$K$ is $\varepsilon$-strongly monotone.
\item
\label{l:3295iii}
Suppose that $\alpha=1$ and $W=\Id$.
Then $K$ is cocoercive with constant $1/(2-\varepsilon)$.
\end{enumerate}
\end{lemma}
\begin{proof}
\ref{l:3295i}:
By the Cauchy--Schwarz inequality,
\begin{align}
(\forall x\in\HH)(\forall y\in\HH)\quad
\scal{x-y}{Kx-Ky}
&=\scal{x-y}{Wx-Wy}-\gamma\scal{x-y}{Bx-By}
\nonumber\\
&\geq\alpha\|x-y\|^2-\gamma\|x-y\|\,\|Bx-By\|
\nonumber\\
&\geq\alpha\|x-y\|^2-\gamma\beta\|x-y\|^2
\nonumber\\
&\geq\varepsilon\|x-y\|^2.
\end{align}

\ref{l:3295iii}:
Since $\gamma B$ is $(1-\varepsilon)$-Lipschitzian,
\cite[Proposition~4.38]{Livre1} entails that
$\gamma B$ is averaged with constant $(2-\varepsilon)/2$.
Hence, since $\gamma B=\Id-K$,
\cite[Proposition~4.39]{Livre1}
implies that $K$ is cocoercive with constant
$1/(2-\varepsilon)$.
\end{proof}

\begin{corollary}
\label{c:fbf}
Let $A\colon\HH\to 2^{\HH}$ be maximally monotone,
let $B\colon\HH\to\HH$ be monotone and $\beta$-Lipschitzian
for some $\beta\in\RPP$, let $(\alpha,\chi)\in\RPP^2$,
and let $\varepsilon\in\left]0,\alpha/(\beta+1)\right[$.
For every $n\in\NN$, let $W_n\colon\HH\to\HH$ be $\alpha$-strongly
monotone and $\chi$-Lipschitzian,
and let
$\gamma_n\in\left[\varepsilon,(\alpha-\varepsilon)/\beta\right]$.
Take $x_0\in\HH$,
let $(\lambda_n)_{n\in\NN}$ be a sequence in $\left]0,2\right[$
such that $0<\inf_{n\in\NN}\lambda_n\leq\sup_{n\in\NN}\lambda_n
<2$, and let $(e_n)_{n\in\NN}$ be a sequence in $\HH$ such that
$e_n\to 0$. Furthermore, let $m\in\NN\smallsetminus\{0\}$
and let $(\mu_{n,j})_{n\in\NN,0\leq j\leq n}$ be a real array
that satisfies the following:
\begin{enumerate}[label={\rm[\alph*]}]
\item
\label{c:fbfa}
For every integer $n>m$ and every
integer $j\in\left[0,n-m-1\right]$, $\mu_{n,j}=0$.
\item
\label{c:fbfb}
For every $n\in\NN$, $\sum_{j=0}^n\mu_{n,j}=1$.
\item
\label{c:fbfc}
$\sup_{n\in\NN}\max_{0\leq j\leq n}|\mu_{n,j}|<\pinf$.
\end{enumerate}
Iterate
\begin{equation}
\label{e:fbf}
\begin{array}{l}
\text{for}\;n=0,1,\ldots\\
\left\lfloor
\begin{array}{l}
\widetilde{x}_n=e_n+\sum_{j=0}^n\mu_{n,j}x_j\\
v_n^*=W_n\widetilde{x}_n-\gamma_nB\widetilde{x}_n\\
y_n=(W_n+\gamma_n A)^{-1}v_n^*\\
y_n^*=\gamma_n^{-1}(v_n^*-W_ny_n)+By_n\\
\text{if}\;\scal{y_n-x_n}{y_n^*}<0\\
\left\lfloor
\begin{array}{l}
x_{n+1}=x_n
+\dfrac{\lambda_n\scal{y_n-x_n}{y_n^*}}{\|y_n^*\|^2}\,y_n^*\\
\end{array}
\right.\\
\text{else}\\
\left\lfloor
\begin{array}{l}
x_{n+1}=x_n.\\
\end{array}
\right.\\[2mm]
\end{array}
\right.\\
\end{array}
\end{equation}
Suppose that $\zer(A+B)\neq\emp$.
Then the following hold:
\begin{enumerate}
\item
\label{c:fbfi}
$\sum_{n\in\NN}\|x_{n+1}-x_n\|^2<\pinf$.
\item
\label{c:fbfii}
$(x_n)_{n\in\NN}$ converges weakly to a point in
$\zer(A+B)$.
\end{enumerate}
\end{corollary}
\begin{proof}
We apply Theorem~\ref{t:1}
with $M=A+B$ and $(\forall n\in\NN)$ $K_n=W_n-\gamma_nB$.
First, \cite[Corollary~20.28]{Livre1} asserts that
$B$ is maximally monotone. Therefore, $M$
is maximally monotone by virtue of
\cite[Corollary~25.5(i)]{Livre1}.
Next, in view of Lemma~\ref{l:3295}\ref{l:3295i},
the kernels $(K_n)_{n\in\NN}$ are
$\varepsilon$-strongly monotone.
Furthermore, the kernels $(K_n)_{n\in\NN}$ are 
Lipschitzian with constant $\alpha+\chi$ since
\begin{align}
(\forall x\in\HH)(\forall y\in\HH)\quad
\|K_nx-K_ny\|
&\leq
\|W_nx-W_ny\|+\gamma_n\|Bx-By\|
\nonumber\\
&\leq\chi\|x-y\|+\frac{\alpha-\varepsilon}{\beta}\beta\|x-y\|
\nonumber\\
&\leq(\alpha+\chi)\|x-y\|.
\end{align}
Therefore, for every $n\in\NN$,
since $K_n+\gamma_n M$ is maximally monotone,
Proposition~\ref{p:20}\ref{p:20ii}\ref{p:20iic}\&%
\ref{p:20i}\ref{p:20ib}
entail that $\ran K_n\subset\ran(K_n+\gamma_n M)$ and
$K_n+\gamma_n M$ is injective.
Let us also observe that \eqref{e:fbf} is a special case
of \eqref{e:fejer15}.

\ref{c:fbfi}:
This follows from Theorem~\ref{t:1}\ref{t:1i}.

\ref{c:fbfii}:
Set $\mu=\sup_{n\in\NN}\max_{0\leq j\leq n}|\mu_{n,j}|$.
For every integer $n>m$, it results from \ref{c:fbfa}
and \ref{c:fbfb} that
\begin{align}
\|\widetilde{x}_n-x_n\|
&=\bigg\|e_n+\Sum_{j=n-m}^n\mu_{n,j}(x_j-x_n)\bigg\|
\nonumber\\
&\leq\|e_n\|+\Sum_{j=n-m}^n|\mu_{n,j}|\|x_j-x_n\|
\nonumber\\
&\leq\|e_n\|+\mu\Sum_{j=n-m}^n\|x_j-x_n\|
\nonumber\\
&=\|e_n\|+\mu\Sum_{j=0}^m\|x_n-x_{n-j}\|.
\end{align}
Therefore, \ref{c:fbfi} and \ref{c:fbfc} imply that
$\widetilde{x}_n-x_n\to 0$.
On the other hand, it follows from Remark~\ref{r:3} that 
condition~\ref{t:1ii}\ref{t:1iib} in Theorem~\ref{t:1} is 
satisfied. Hence, the conclusion follows from 
Theorem~\ref{t:1}\ref{t:1ii}.
\end{proof}

Next, we recover Tseng's forward-backward-forward algorithm
\cite{Livre1,Tsen00}.

\begin{corollary}
\label{c:fbf2}
Let $A\colon\HH\to 2^{\HH}$ be maximally monotone,
let $B\colon\HH\to\HH$ be monotone and $\beta$-Lipschitzian
for some $\beta\in\RPP$. Suppose that $\zer(A+B)\neq\emp$,
take $x_0\in\HH$,
let $\varepsilon\in\left]0,1/(\beta+1)\right[$,
and let $(\gamma_n)_{n\in\NN}$ be a sequence in
$\left[\varepsilon,(1-\varepsilon)/\beta\right]$.
Iterate
\begin{equation}
\label{e:8732}
\begin{array}{l}
\text{for}\;n=0,1,\ldots\\
\left\lfloor
\begin{array}{l}
v_n^*=\gamma_nBx_n\\
y_n=J_{\gamma_nA}(x_n-v_n^*)\\
x_{n+1}=y_n-\gamma_nBy_n+v_n^*.
\end{array}
\right.\\
\end{array}
\end{equation}
Then $(x_n)_{n\in\NN}$ converges weakly to a point in
$\zer(A+B)$.
\end{corollary}
\begin{proof}
We apply Theorem~\ref{t:1} with $M=A+B$ and
$(\forall n\in\NN)$ $K_n=\Id-\gamma_nB$
and $\widetilde{x}_n=x_n$.
Note that the kernels
$(K_n)_{n\in\NN}$ are cocoercive with constant
$1/(2-\varepsilon)$ by virtue of Lemma~\ref{l:3295}\ref{l:3295iii}.
Moreover, using Lemma~\ref{l:3295}\ref{l:3295i},
we deduce that the kernels $(K_n)_{n\in\NN}$
are strongly monotone with constant $\varepsilon$.
Thus, for every $n\in\NN$, since
$K_n+\gamma_n M=\Id+\gamma_n A$ is maximally monotone,
Proposition~\ref{p:20}\ref{p:20ii}%
\ref{p:20iic}\&\ref{p:20i}\ref{p:20ib}
assert that $\ran K_n\subset\ran(K_n+\gamma_nM)$ and
$K_n+\gamma_n M$ is injective. Now set
\begin{equation}
\label{e:1172}
(\forall n\in\NN)\quad
y_n^*=\gamma_n^{-1}(K_nx_n-K_ny_n)
\quad\text{and}\quad
\lambda_n=
\begin{cases}
\dfrac{\gamma_n\|y_n^*\|^2}{\scal{x_n-y_n}{y_n^*}},
&\text{if}\;\scal{x_n-y_n}{y_n^*}>0;\\
\varepsilon,&\text{otherwise}.
\end{cases}
\end{equation}
Fix $n\in\NN$. Then, by strong monotonicity
of $K_n$ and the Cauchy--Schwarz inequality,
\begin{equation}
\varepsilon\|x_n-y_n\|^2
\leq\scal{x_n-y_n}{K_nx_n-K_ny_n}
\leq\|x_n-y_n\|\,\|K_nx_n-K_ny_n\|.
\end{equation}
This implies that
$\scal{x_n-y_n}{y_n^*}
=\gamma_n^{-1}\scal{x_n-y_n}{K_nx_n-K_ny_n}
\leq\gamma_n^{-1}\|x_n-y_n\|\,\|K_nx_n-K_ny_n\|
\leq(\varepsilon\gamma_n)^{-1}\|K_nx_n-K_ny_n\|^2
=\varepsilon^{-1}\gamma_n\|y_n^*\|^2$
and therefore that $\lambda_n\geq\varepsilon$.
In addition, by cocoercivity of $K_n$,
$\gamma_n\|y_n^*\|^2
=\gamma_n^{-1}\|K_nx_n-K_ny_n\|^2
\leq(2-\varepsilon)\gamma_n^{-1}\scal{x_n-y_n}{K_nx_n-K_ny_n}
=(2-\varepsilon)\scal{x_n-y_n}{y_n^*}$
and thus $\lambda_n\leq 2-\varepsilon$.
Next, we derive from \eqref{e:8732} that
$y_n=J^{K_n}_{\gamma_n M}x_n$.
If $\scal{x_n-y_n}{y_n^*}>0$,
then \eqref{e:8732} and \eqref{e:1172} yield
$x_{n+1}=x_n-\gamma_ny_n^*
=x_n+\lambda_n\scal{y_n-x_n}{y_n^*}y_n^*/\|y_n^*\|^2$.
Otherwise,
$\scal{x_n-y_n}{y_n^*}\leq 0$ and
the cocoercivity of $K_n$ yields 
$\|y_n^*\|^2=\gamma_n^{-2}\|K_nx_n-K_ny_n\|^2
\leq(2-\varepsilon)\gamma_n^{-2}\scal{x_n-y_n}{K_nx_n-K_ny_n}
\leq 0$. 
Hence, $y_n^*=0$ and we therefore deduce from
\eqref{e:8732} that $x_{n+1}=x_n$.
Thus, \eqref{e:8732} is an instance of \eqref{e:fejer15}.
Next, condition~\ref{t:1ii}\ref{t:1iia} in Theorem~\ref{t:1}
is trivially satisfied and, in view of Remark~\ref{r:3},
condition~\ref{t:1ii}\ref{t:1iib} in Theorem~\ref{t:1} is
also fulfilled.
\end{proof}

We conclude this section by further illustrating the effectiveness
of warped resolvent iterations by designing a new method to solve an 
intricate system of monotone inclusions and its dual. We are not 
aware of a splitting method that could handle such a formulation
with a comparable level of flexibility. Special cases of this
system appear in \cite{Siop14,Bot13d,MaPr18,John18}.

\begin{problem}
\label{prob:16}
Let $(\GG_i)_{i\in I}$ and 
$(\KK_j)_{j\in J}$ be finite families of real Hilbert spaces.
For every $i\in I$ and 
$j\in J$, let $A_i\colon\GG_i\to 2^{\GG_i}$ and 
$B_j\colon\KK_j\to 2^{\KK_j}$ 
be maximally monotone,
let $C_i\colon\GG_i\to\GG_i$ be monotone and
$\mu_i$-Lipschitzian for some $\mu_i\in\RPP$,
let $D_j\colon\KK_j\to\KK_j$ be monotone and
$\nu_j$-Lipschitzian for some $\nu_j\in\RPP$,
let $L_{ji}\in\BL(\GG_i,\KK_j)$,
let $s_i^*\in\GG_i$, and let $r_j\in\KK_j$. 
Consider the system of coupled inclusions
\begin{multline}
\label{e:0371}
\text{find}\;\:(x_i)_{i\in I}\in\bigtimes_{i\in I}\GG_i
\;\:\text{such that}\\
(\forall i\in I)\quad
s_i^*\in A_ix_i+\Sum_{j\in J}L_{ji}^*
\bigg((B_j+D_j)\bigg(\Sum_{k\in I}L_{jk}
x_k-r_j\bigg)\bigg)+C_ix_i,
\end{multline}
its dual problem 
\begin{multline}
\label{e:7301}
\text{find}\;\:(v^*_j)_{j\in J}\in\bigtimes_{j\in J}
\KK_j\;\:\text{such that}\\\bigg(\exi(x_i)_{i\in I}\in
\bigtimes_{i\in I}\GG_i\bigg)(\forall i\in I)(\forall j\in J)\quad
\begin{cases}
s^*_i-\Sum_{k\in J}L_{ki}^*v_k^*\in A_ix_i+C_ix_i\\
v_j^*\in(B_j+D_j)\bigg(\Sum_{k\in I}L_{jk}x_k-r_j\bigg),
\end{cases}
\end{multline}
and the associated Kuhn--Tucker set
\begin{multline}
\label{e:6464}
Z=\bigg\{\big((x_i)_{i\in I},
(v^*_j)_{j\in J}\big)\;\bigg |\;
(\forall i\in I)\;\;x_i\in\GG_i\;\;\text{and}\;\;
s^*_i-\sum_{k\in J}L_{ki}^*v_k^*\in
A_ix_i+C_ix_i,\\\text{and}\:\;
(\forall j\in J)\;\;v_j^*\in\KK_j\;\;\text{and}\;\;
\sum_{k\in I}L_{jk}x_k-r_j\in
(B_j+D_j)^{-1}v_j^*
\bigg\}.
\end{multline}
We denote by $\mathscr{P}$ and $\mathscr{D}$
the sets of solutions to \eqref{e:0371} and \eqref{e:7301},
respectively. The problem is to find a point in $Z$.
\end{problem}

\begin{corollary}
\label{c:7}
Consider the setting of Problem~\ref{prob:16}.
For every $i\in I$ and every $j\in J$,
let $(\alpha_i,\chi_i,\beta_j,\kappa_j)\in\RPP^4$,
let $\varepsilon_i\in\left]0,\alpha_i/(\mu_i+1)\right[$,
let $\delta_j\in\left]0,\beta_j/(\nu_j+1)\right[$,
let $(F_{i,n})_{n\in\NN}$ be operators from $\GG_i$
to $\GG_i$ that are $\alpha_i$-strongly monotone
and $\chi_i$-Lipschitzian,
let $(W_{j,n})_{n\in\NN}$ be operators from $\KK_j$
to $\KK_j$ that are $\beta_j$-strongly monotone
and $\kappa_j$-Lipschitzian; in addition,
let $(\gamma_{i,n})_{n\in\NN}$ and $(\tau_{j,n})_{n\in\NN}$
be sequences in
$\left[\varepsilon_i,(\alpha_i-\varepsilon_i)/\mu_i\right]$
and $\left[\delta_j,(\beta_j-\delta_j)/\nu_j\right]$,
respectively. Suppose that $Z\neq\emp$ and that
\begin{equation}
\GG=\bigtimes_{i\in I}\GG_i,
\quad
\KK=\bigtimes_{j\in J}\KK_j,
\quad\text{and}\quad
\HH=\GG\times\KK\times\KK.
\end{equation}
Let $((x_{i,0})_{i\in I},(y_{j,0})_{j\in J},
(v_{j,0}^*)_{j\in J})$ and 
$((\widetilde{x}_{i,n})_{i\in I},
(\widetilde{y}_{j,n})_{j\in J},
(\widetilde{v}_{j,n}^*)_{j\in J})_{n\in\NN}$ be in $\HH$,
and let $(\lambda_n)_{n\in\NN}$ be a sequence in
$\left]0,2\right[$ such that
$0<\inf_{n\in\NN}\lambda_n\leq\sup_{n\in\NN}\lambda_n<2$.
Iterate
\pagebreak[1]
\begin{equation}
\label{e:1089}
\begin{array}{l}
\text{for}\;n=0,1,\ldots\\
\left\lfloor
\begin{array}{l}
\text{for every}\;i\in I\\
\left\lfloor
\begin{array}{l}
l_{i,n}^*=F_{i,n}\widetilde{x}_{i,n}-\gamma_{i,n}C_i
\widetilde{x}_{i,n}
-\gamma_{i,n}\sum_{j\in J}L_{ji}^*\widetilde{v}_{j,n}^*\\
a_{i,n}=\big(F_{i,n}+\gamma_{i,n}A_i\big)^{-1}
(l_{i,n}^*+\gamma_{i,n}s_i^*)\\
o_{i,n}^*=\gamma_{i,n}^{-1}(l_{i,n}^*
-F_{i,n}a_{i,n})+C_ia_{i,n}\\
\end{array}
\right.
\\
\text{for every}\;j\in J\\
\left\lfloor
\begin{array}{l}
\displaystyle
t^*_{j,n}=W_{j,n}\widetilde{y}_{j,n}-\tau_{j,n}D_j\widetilde{y}_{j,n}
+\tau_{j,n}\widetilde{v}_{j,n}^*\\
b_{j,n}=\big(W_{j,n}+\tau_{j,n}B_j\big)^{-1}t^*_{j,n}\\
f_{j,n}^*=\tau_{j,n}^{-1}(t_{j,n}^*
-W_{j,n}b_{j,n})+D_jb_{j,n}
\\
c_{j,n}=\sum_{i\in I}L_{ji}\widetilde{x}_{i,n}
-\widetilde{y}_{j,n}+\widetilde{v}_{j,n}^*-r_j\\
\end{array}
\right.
\\
\text{for every}\;i\in I\\
\left\lfloor
\begin{array}{l}
a_{i,n}^*=o_{i,n}^*+\sum_{j\in J}L^*_{ji}c_{j,n}
\end{array}
\right.
\\
\text{for every}\;j\in J\\
\left\lfloor
\begin{array}{l}
b_{j,n}^*=f_{j,n}^*-c_{j,n}
\\
c_{j,n}^*=r_j+b_{j,n}-\sum_{i\in I}L_{ji}a_{i,n}
\end{array}
\right.
\\
\sigma_n=\sum_{i\in I}\|a_{i,n}^*\|^2
+\sum_{j\in J}\big(\|b_{j,n}^*\|^2+\|c_{j,n}^*\|^2\big)
\\
\theta_n=\sum_{i\in I}\scal{a_{i,n}-x_{i,n}}{a_{i,n}^*}
+\sum_{j\in J}\big(\scal{b_{j,n}-y_{j,n}}{b_{j,n}^*}
+\scal{c_{j,n}-v_{j,n}^*}{c_{j,n}^*}
\big)\\
\text{if}\;\theta_n<0\\
\left\lfloor
\begin{array}{l}
\rho_n=\lambda_n\theta_n/\sigma_n
\end{array}
\right.\\
\text{else}\\
\left\lfloor
\begin{array}{l}
\rho_n=0
\end{array}
\right.\\
\text{for every}\;i\in I\\
\left\lfloor
\begin{array}{l}
x_{i,n+1}=x_{i,n}+\rho_na_{i,n}^*
\end{array}
\right.\\
\text{for every}\;j\in J\\
\left\lfloor
\begin{array}{l}
y_{j,n+1}=y_{j,n}+\rho_nb_{j,n}^*\\
v_{j,n+1}^*=v_{j,n}^*+\rho_nc_{j,n}^*.
\end{array}
\right.\\[4mm]
\end{array}
\right.
\end{array}
\end{equation}
Suppose that
\begin{equation}
\label{e:fr}
(\forall i\in I)(\forall j\in J)
\quad\widetilde{x}_{i,n}-x_{i,n}\to 0,
\quad\widetilde{y}_{j,n}-y_{j,n}\to 0,
\quad\text{and}\quad
\widetilde{v}^*_{j,n}-v^*_{j,n}\to 0.
\end{equation}
Set
$(\forall n\in\NN)$ $x_n=(x_{i,n})_{i\in I}$ and 
$v_n^*=(v_{j,n}^*)_{j\in J}$.
Then 
$(x_n)_{n\in\NN}$ converges weakly to a point
$\overline{x}\in\mathscr{P}$,
$(v_n^*)_{n\in\NN}$ converges weakly to a point
$\overline{v}^*\in\mathscr{D}$,
and $(\overline{x},\overline{v}^*)\in Z$.
\end{corollary}
\begin{proof}
Define
\begin{equation}
\label{e:9906}
\begin{cases}
\displaystyle
A\colon\GG\to 2^{\GG}\colon
(x_i)_{i\in I}
\mapsto\bigtimes_{i\in I}(A_ix_i+C_ix_i)\\
\displaystyle
B\colon\KK\to 2^{\KK}\colon
(y_j)_{j\in J}
\mapsto\bigtimes_{j\in J}(B_jy_j+D_jy_j)\\
L\colon\GG\to\KK\colon
(x_i)_{i\in I}
\mapsto\bigg(\Sum_{i\in I}L_{ji}x_i\bigg)_{j\in J}\\
s^*=(s_i^*)_{i\in I}\quad\text{and}\quad
r=(r_j)_{j\in J}.
\end{cases}
\end{equation}
We observe that
\begin{equation}
\label{e:6789}
L^*\colon\KK\to\GG\colon
(v_j^*)_{j\in J}\mapsto
\Bigg(\Sum_{j\in J}L_{ji}^*v_j^*\Bigg)_{i\in I}.
\end{equation}
In the light of \cite[Proposition~20.23]{Livre1},
$A$ and $B$ are maximally monotone.
On the other hand, we deduce from \eqref{e:6464}, \eqref{e:9906},
and \eqref{e:6789} that
\begin{equation}
Z=\menge{(x,v^*)\in\GG\times\KK}{s^*-L^*v^*\in Ax\;\text{and}\;
Lx-r\in B^{-1}v^*}.
\end{equation}
Define
\begin{equation}
\label{e:4095}
M\colon\HH\to 2^{\HH}\colon
(x,y,v^*)\mapsto
(-s^*+Ax+L^*v^*)\times(By-v^*)\times\{r-Lx+y\}.
\end{equation}
Lemma~\ref{l:5498}\ref{l:5498i} entails that $M$ is maximally
monotone. Furthermore, since $Z\neq\emp$,
Lemma~\ref{l:5498}\ref{l:5498iii} yields $\zer M\neq\emp$.
Next, set 
\begin{equation}
S\colon\HH\to\HH\colon
(x,y,v^*)\mapsto(-L^*v^*,v^*,Lx-y)
\end{equation}
and, for every $n\in\NN$, 
\begin{equation}
\label{e:6095}
K_n\colon\HH\to\HH\colon
(x,y,v^*)\mapsto
\Big(\big(\gamma_{i,n}^{-1}F_{i,n}x_i-C_ix_i\big)_{i\in I}
-L^*v^*,
\big(\tau_{j,n}^{-1}W_{j,n}y_j-D_jy_j\big)_{j\in J}+v^*,
Lx-y+v^*\Big)
\end{equation}
and 
\begin{equation}
T_n\colon\HH\to\HH\colon
(x,y,v^*)\mapsto
\Big(\big(\gamma_{i,n}^{-1}F_{i,n}x_i-C_ix_i\big)_{i\in I},
\big(\tau_{j,n}^{-1}W_{j,n}y_j-D_jy_j\big)_{j\in J},v^*\Big).
\end{equation}
For every $i\in I$ and every $n\in\NN$,
using the facts that $C_i$ is $\mu_i$-Lipschitzian, that
$F_{i,n}$ is $\alpha_i$-strongly monotone, and that
$\gamma_{i,n}\in\left[\varepsilon_i,
(\alpha_i-\varepsilon_i)/\mu_i\right]$,
Lemma~\ref{l:3295}\ref{l:3295i} implies that
$F_{i,n}-\gamma_{i,n}C_i$ is $\varepsilon_i$-strongly monotone
and therefore, since $\gamma_{i,n}^{-1}\geq
\mu_i/(\alpha_i-\varepsilon_i)$,
it follows that $\gamma_{i,n}^{-1}F_{i,n}-C_i$
is strongly monotone with constant
$\varepsilon_i\mu_i/(\alpha_i-\varepsilon_i)$.
Likewise, for every $j\in J$ and every $n\in\NN$,
$\tau_{j,n}^{-1}W_{j,n}-D_j$ is strongly monotone with constant
$\delta_j\nu_j/(\beta_j-\delta_j)$.
Thus, upon setting
\begin{equation}
\vartheta=
\min\Bigg\{
\min_{i\in
I}\frac{\varepsilon_i\mu_i}{\alpha_i-\varepsilon_i},
\min_{j\in
J}\frac{\delta_j\nu_j}{\beta_j-\delta_j},1
\Bigg\},
\end{equation}
we get
\begin{align}
&(\forall n\in\NN)\big(\forall(x,y,v^*)\in\HH\big)
\big(\forall(a,b,c^*)\in\HH\big)
\nonumber\\
&\hskip 10mm
\sscal{(x,y,v^*)-(a,b,c^*)}{T_n(x,y,v^*)-T_n(a,b,c^*)}
\nonumber\\
&\hskip 20mm
=\Sum_{i\in I}
\sscal{x_i-a_i}{\big(\gamma_{i,n}^{-1}F_{i,n}x_i-C_ix_i\big)
-\big(\gamma_{i,n}^{-1}F_{i,n}a_i-C_ia_i\big)}
\nonumber\\
&\hskip 20mm
\quad\;+\Sum_{j\in J}\sscal{y_j-b_j}{\big(
\tau_{j,n}^{-1}W_{j,n}y_j-D_jy_j\big)-
\big(\tau_{j,n}^{-1}W_{j,n}b_j-D_jb_j\big)}
+\|v^*-c^*\|^2
\nonumber\\
&\hskip 20mm
\geq\vartheta\Sum_{i\in I}\|x_i-a_i\|^2
+\vartheta\Sum_{j\in J}\|y_j-b_j\|^2
+\vartheta\|v^*-c^*\|^2
\nonumber\\
&\hskip 20mm
=\vartheta\|(x,y,v^*)-(a,b,c^*)\|^2.
\end{align}
Hence, the operators
$(T_n)_{n\in\NN}$ are $\vartheta$-strongly monotone.
However, $S$ is linear, bounded, and $S^*=-S$.
It follows that the kernels
$(K_n)_{n\in\NN}=(T_n+S)_{n\in\NN}$ are $\vartheta$-strongly
monotone. Now, for every $i\in I$ and every $n\in\NN$,
since $\gamma_{i,n}^{-1}F_{i,n}$ is Lipschitzian with constant
$\chi_i/\varepsilon_i$, we deduce that
$\gamma_{i,n}^{-1}F_{i,n}-C_i$ is Lipschitzian with constant
$\chi_i/\varepsilon_i+\mu_i$.
Likewise, for every $j\in J$ and every $n\in\NN$,
$\tau_{j,n}^{-1}W_{j,n}-D_j$ is Lipschitzian with constant
$\kappa_j/\delta_j+\nu_j$. Hence, upon setting
\begin{equation}
\eta=\max\Big\{
\max_{i\in I}\{\chi_i/\varepsilon_i+\mu_i\},
\max_{j\in J}\{\kappa_j/\delta_j+\nu_j\},1\Big\},
\end{equation}
we obtain
\begin{align}
&\hskip -40mm
(\forall n\in\NN)\big(\forall(x,y,v^*)\in\HH\big)
\big(\forall(a,b,c^*)\in\HH\big)
\quad\|T_n(x,y,v^*)-T_n(a,b,c^*)\|^2
\nonumber\\
\hskip 35mm&
=\Sum_{i\in I}\big\|\big(\gamma_{i,n}^{-1}F_{i,n}x_i-C_ix_i\big)
-\big(\gamma_{i,n}^{-1}F_{i,n}a_i-C_ia_i\big)\big\|^2
\nonumber\\
&
\quad\;+\Sum_{j\in J}\big\| 
\big(\tau_{j,n}^{-1}W_{j,n}y_j-D_jy_j\big)-
\big(\tau_{j,n}^{-1}W_{j,n}b_j-D_jb_j\big)\big\|^2
+\|v^*-c^*\|^2
\nonumber\\
&
\leq\eta^2\Sum_{i\in I}\|x_i-a_i\|^2
+\eta^2\Sum_{j\in J}\|y_j-b_j\|^2
+\eta^2\|v^*-c^*\|^2
\nonumber\\
&
=\eta^2\|(x,y,v^*)-(a,b,c^*)\|^2.
\end{align}
This implies that the operators
$(T_n)_{n\in\NN}$ are $\eta$-Lipschitzian.
On the other hand, $S$ is Lipschitzian with constant
$\|S\|$. Altogether, the kernels
$(K_n)_{n\in\NN}$ are Lipschitzian with constant
$\eta+\|S\|$. In turn, using
Proposition~\ref{p:20}\ref{p:20ii}\ref{p:20iic}%
\&\ref{p:20i}\ref{p:20ib},
we infer that, for every $n\in\NN$, $\ran K_n\subset\ran(K_n+M)$
and $K_n+M$ is injective.
Now set 
\begin{multline}
(\forall n\in\NN)\quad
p_n=\big((x_{i,n})_{i\in I},
(y_{j,n})_{j\in J},(v^*_{j,n})_{j\in J}\big),\quad
\widetilde{p}_n=\big((\widetilde{x}_{i,n})_{i\in I},
(\widetilde{y}_{j,n})_{j\in J},
(\widetilde{v}^*_{j,n})_{j\in J}\big),\\
q_n=\big((a_{i,n})_{i\in I},
(b_{j,n})_{j\in J},(c_{j,n})_{j\in J}\big),\;
\quad\text{and}\quad
q_n^*=\big((a_{i,n}^*)_{i\in I},(b_{j,n}^*)_{j\in J},
(c_{j,n}^*)_{j\in J}\big).
\end{multline}
In view of \eqref{e:6095}, \eqref{e:4095}, \eqref{e:9906},
and \eqref{e:6789},
we deduce that \eqref{e:1089} assumes the form
\begin{equation}
\label{e:1207}
\begin{array}{l}
\text{for}\;n=0,1,\ldots\\
\left\lfloor
\begin{array}{l}
q_n=J_{M}^{K_n}\widetilde{p}_n\\
q_n^*=K_n\widetilde{p}_n-K_nq_n\\
\text{if}\;\scal{q_n-p_n}{q_n^*}<0\\
\left\lfloor
\begin{array}{l}
p_{n+1}=p_n+\dfrac{\lambda_n\sscal{q_n-p_n}
{q_n^*}}{\|q_n^*\|^2}\,q_n^*
\end{array}
\right.\\
\text{else}\\
\left\lfloor
\begin{array}{l}
p_{n+1}=p_n.\\
\end{array}
\right.\\[2mm]
\end{array}
\right.\\
\end{array}
\end{equation}
In addition, \eqref{e:fr} implies that $\widetilde{p}_n-p_n\to 0$.
Altogether, in the light of Theorem~\ref{t:1} and
Remark~\ref{r:3}, there exists
$(\overline{x},\overline{y},\overline{v}^*)\in\zer M$
such that $p_n\weakly(\overline{x},\overline{y},\overline{v}^*)$.
It follows that
$x_n\weakly\overline{x}$ and $v_n^*\weakly\overline{v}^*$. 
Further,
we conclude by using Lemma~\ref{l:5498}\ref{l:5498ii} that
$\overline{x}\in\mathscr{P}$, $\overline{v}^*\in\mathscr{D}$,
and $(\overline{x},\overline{v}^*)\in Z$.
\end{proof}

\end{document}